\documentclass[a4paper,12pt]{article}
\usepackage[dvips]{graphicx}
\usepackage{amsmath}
\usepackage{amsthm}
\usepackage{amssymb}

\title{Sliding possibility of the Julia sets}
\author{Hiromichi Nakayama and Takuya Takahashi}
\newtheorem{thm}{Theorem}
\newtheorem{lemma}{Lemma}

\theoremstyle{remark}

\renewcommand{\qedsymbol}{$\blacksquare$}

\renewcommand{\phi}{\varphi}
\renewcommand{\epsilon}{\varepsilon}

\newcommand{\RR}{\mathbf{R}}

\newcommand{\CC}{\mathbf{C}}

\DeclareMathOperator{\diam}{diam}
\DeclareMathOperator{\id}{id}

\date{}
\begin{document}
\maketitle

A.\,Sannami constructed an example of the $C^{\infty}$-Cantor set embedded in the real line $\RR$ whose difference set has a positive measure in \cite{San}, which was an answer of the question given by J.\ Palis (\cite{Pa}).
In this paper, we generalize the definition of the difference sets for sets of the two dimensional Euclidean space $\RR^2$ and estimate the measure of the difference sets of the Julia sets.

Here the difference set for two subsets $X$ and $Y$ of $\RR^2$ is defined by  
$$X-Y=\{x - y \in \RR^2 \,;\, x \in X, y \in Y\}.$$
For a vector $\mu$ of $X - X$, there are two points $x$ and $y$ of $X$ satisfying $x=y+\mu$.
Thus the slided set $X+\mu$ intersects $X$ itself.
In this point of view, the measure of $X-X$ is closely related to the problem whether the set $X$  can slide easily.

 Let ${\cal C}$ be a Cantor set in $\RR$ whose difference set has a positive measure (\cite{San}).
Then we have $m(({\cal C}\times {\cal C})-({\cal C} \times {\cal C}))=m(({\cal C}-{\cal C})\times ({\cal C}-{\cal C}))=m({\cal C}-{\cal C})^2>0$ for the Lebesgue measure $m$.
Thus the measure of the $2$-dimensional difference set is not always zero.
In this paper, we consider the condition when the measures of the difference sets of the Julia sets vanish.

Let $c$ be an element of $\CC$.
We define the quadratic map $Q_c:\CC \to \CC$ by $Q_c(z)=z^2+c$.
Let $D=\{ z\,;\, |z|\leq |c|\}$.
The filled Julia set is the set of the bounded orbits of $Q_c$, denoted by $K_c$, and the Julia set $J_c$ is the boundary of $K_c$.
In the following, we assume that $|c|>2$. 
Then it was shown that $K_c$ is a Cantor set and $J_c$ coincides with $K_c$ (\cite{De}).
Furthermore, $J_c=\bigcap_{n \geq 0} {Q_c}^{-n}(D)$.

\begin{thm}
\label{thm1}
If $|c|>3+\sqrt{3}$, then the Lebesgue measure $m$ of the difference set of the Julia set $J_c$ vanishes $($i.\,e. $m(J_c-J_c)=0)$.
\end{thm}

The authors thank to S.\,Matsumoto who gave the attention on the difference sets to the author.

\section{Preparation for the Julia sets}

Let us define $F_+:\CC\to \CC$ by $F_+(r e^{i\theta})=\sqrt{r} e^{\frac{i\theta}2}$ for $0\leq \theta < 2\pi$ and $F_-:\CC\to \CC$ by $F_-(r e^{i\theta})=- \sqrt{r} e^{\frac{i\theta}2}$ for $0\leq \theta < 2\pi$.
For the map $L(z)=z-c$,  the composition $F_+\cdot L$ is denoted by $G_0$ and the composition $F_- \cdot L$ is denoted by $G_1$.
Then $Q_c\cdot G_0=\id$ and $Q_c\cdot G_1=\id$.
Thus $G_0$ and $G_1$ are the partial inverse maps of $Q_c$.

When $|c|>2$, $G_0(D)\subset D$ and $G_1(D)\subset D$.
Let $I_0=G_0(D)$ and $I_1=G_1(D)$.
Then $I_0\cup I_1$ is the set contained in the well-known $8$-figured set.
Let $\{ s_0, s_1, \cdots,  s_n\}$ be a sequence consisting of the numbers $0$ and $1$.
Let $I{s_0 s_1\cdots s_n}=\{ z \in D\,;\, z\in I{s_0}, Q_c(z)\in I{s_1},  \cdots, {Q_c}^n(z) \in I{s_n} \}$.
Then $I{s_0s_1\cdots s_n}=G_{s_0}\cdot G_{s_1} \cdot \cdots \cdot G_{s_n}(D)$ and ${Q_c}^{-n}(D)=\bigcup_{\{s_0,s_1,\cdots, s_n\}} I{s_0s_1\cdots s_n}$.
The set $\bigcap_{n\geq 0} I{s_0\dots s_n}$ is called the component of the Julia set $J_c$.
Then $J_c$ is the union of the components.
The diameter of $I{s_0s_1\cdots s_n}$ is defined by
$$\diam I{s_0s_1\cdots s_n}=\max\{|z-w|\,;\,z, w\in I{s_0s_1\cdots s_n} \}.$$

We define the sequences $\{R_n\}_{n=1,2,\cdots}$ and $\{r_n\}_{n=1,2,\cdots}$ by $R_{n+1}=\sqrt{|c|+R_n}$, $R_1=\sqrt{2 | c|}$ and $r_{n+1}=\sqrt{|c|-R_n}$ and $r_1=0$.
Then $r_n\leq |z| \leq R_n$ for $z\in {Q_c}^{-n}(D)$.
When $|c|>2$, $\{R_n\}$ is monotone decreasing and $\{r_n\}$ is monotone increasing by induction.

\begin{lemma}
\label{lem1}
$\diam I{s_0s_1\cdots s_n} < \frac{1}{\sqrt{2}r_{n+1}} \diam I{s_1\cdots s_n}$ for $n \geq 1$.
In particular, $$\diam I{s_0s_1\cdots s_n} < 2^{-\frac{n}{2}} \frac{1}{r_{n+1}r_{n}\cdots r_2}\diam I_0.$$
\end{lemma}

\begin{proof}
Let $z'$ and $w'$ be points of $I{s_0s_1\cdots s_n}$ such that $|z'-w'|=\diam I{s_0s_1\cdots s_n}$.
By definition, $I{s_0s_1\cdots s_n}=G_{s_0}(I{s_1\cdots s_n})$, and thus $Q_c(I{s_0s_1\cdots s_n})=I{s_1\cdots s_n}$.
Therefore, 
\begin{eqnarray*}
& &\diam I{s_1\cdots s_n}\\
&\geq& |Q_c(z')-Q_c(w')|\\
&=& |({z'}^2+c)-({w'}^2+c)|\\
&=&|{z'}^2-{w'}^2|\\
&=&|z'-w'|\ |z'+w'|.
\end{eqnarray*}
Let $z=Q_c(z')$ and $w=Q_c(w')$.
Then $z'=G_{s_0}(z)$ and $w'=G_{s_0}(w)$ are contained in $I{s_0}$.
Thus the angle between $z'$ and $w'$ is less than $\pi/2$.
Since $\angle (z',w') < \pi/2$, $|z'|\geq r_n$ and $|w'|\geq r_n$, we have $$|z'+w'|=\sqrt{|z'|^2+|w'|^2+2|z'|\,|w'|\cos \angle (z', w')} > \sqrt{|z'|^2+|w'|^2} \geq \sqrt{2}r_{n+1}.$$
Thus 
\begin{eqnarray*}
& & \diam I{s_0s_1\cdots s_n}\\
& < & \frac{1}{\sqrt{2} r_{n+1}} \diam I{s_1\dots s_n} \\
& < & \left( \frac{1}{\sqrt{2} r_{n+1}} \right) \cdots \left( \frac{1}{\sqrt{2} r_2}\right) \diam I{s_n}\\
&=&\left( \frac{1}{\sqrt{2}}\right)^n \frac{1}{r_{n+1}r_n \cdots r_2} \diam I_0.
\end{eqnarray*}
\end{proof}

Let $K_n$ denote $2^{-\frac{n}{2}} \frac{1}{r_{n+1}r_n\cdots r_2}\diam I_0$ in Lemma \ref{lem1}.
Thus $\diam I{s_0s_1\cdots s_n}< K_n$ for $n \geq 1$.

\section{Estimate of the measure of the difference sets}

\begin{lemma}
\label{lem2}
For any $I{s_0s_1\cdots, s_n}$, there is a closed disk $D{s_0s_1\cdots s_n}$ containing $I{s_0s_1\cdots s_n}$ such that the radius of $D{s_0s_1\cdots s_n}$ is equal to $\frac{\sqrt{3}}{2}\diam I{s_0s_1\cdots s_n}$.
\end{lemma}

\begin{proof}
Let $x$ and $y$ be points of $I{s_0s_1\cdots s_n}$ whose distance is equal to $\diam I{s_0s_1\cdots s_n}$.
For any point $z$ of $I{s_0s_1\cdots s_n}$, we have $|z-x|\leq |x-y|$ and $|z-y|\leq |x-y|$.
Thus $z$ is contained in the closed disk with the center $x$ whose radius is $|x-y|$, and furthermore $z$ is contained in the closed disk with the center $y$ whose radius is $|x-y|$.
The intersection of these closed disks is contained in the closed disk whose center is the middle point of $x$ and $y$ and radius is $\frac{\sqrt{3}}{2} |x-y|$.
Then this disk satisfies the conditions of $D{s_0s_1\cdots s_n}$.
\end{proof}

The above estimate of the radius of $D{s_0s_1\cdots s_n}$ can be made more strict, but it does not improve the condition of the main theorem in the following proof.

\begin{lemma}
\label{lem3}
Let $D_1$ and $D_2$ be closed disks whose radii are equal to $R$.
Let $O_1$ and $O_2$ denote the centers of $D_1$ and $D_2$ respectively.
Then $D_2-D_1$ is the closed disk $\{ v\,;\, |v-\overrightarrow{O_1O_2}|\leq 2R\}$, and thus the center is $\overrightarrow{O_1O_2}$ and the radius is $2R$.
\end{lemma}

\begin{proof}
Let $A$ denote the middle point of $O_1$ and $O_2$.
Denote by $r$ the length between $O_1$ and $O_2$.
Let $\ell_1$ denote the straight line passing through $O_1$ and $O_2$.
Let $O_3$ denote the point in $\ell_1$ satisfying that the distance between $O_3$ and $A$ is equal to $r$ and $O_3$ is closer to $O_2$ than $O_1$.
Denote by $D_3$ the closed disk with the center $O_3$ whose radius is $2R$.
Then $D_3$ and $D_2$ are similar, and the ratio is $2$.

Let $\ell_2$ be a straight line passing through $A$ and intersecting $D_2$.
Let $P$ denote the point among the intersection points of $\ell_2$ and $\partial D_1$ which is further from $D_2$ (Fig \ref{figsim}).
The other intersection point of $\ell_2$ and $\partial D_1$ is denoted by $S$ if $\ell_2 \cap \partial D_1$ is not a single point.
If $\ell_2 \cap \partial D_1$ is a single point, then the point $S$ is given by $P$.
Let $U$ denote the point of $\ell_2 \cap \partial D_2$ which is further from $D_1$.
The other point of $\ell_2 \cap \partial D_2$ is denoted by $T$ ($T=U$ if $\ell_2 \cap \partial D_2$ is a single point).
Furthermore, let $W$ denote the intersection point of $\ell_2$ and $\partial D_3$ further from $D_1$.
The other intersection point is denoted by $V$ ($V=W$ when $\ell_2 \cap \partial D_3$ is a single point).
\begin{figure}
\label{figsim}
\begin{center}
\includegraphics[width=\linewidth,clip]{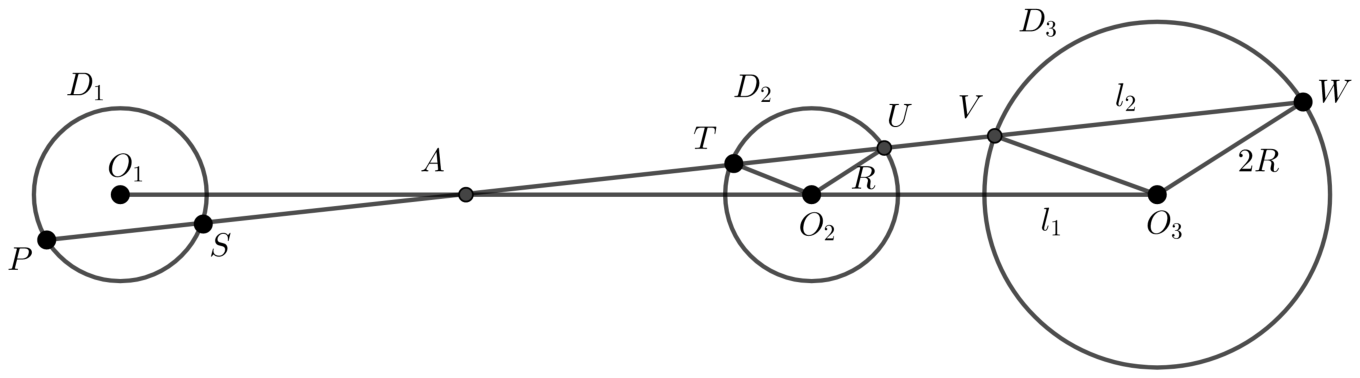}
\end{center}
\caption{}
\end{figure}

We take the orientation of $\ell_2$ so that the direction from $A$ to $U$ is positive.
Then the maximum vector of $D_2-D_1$ in $\ell_2$ with respect to this orientation is $\overrightarrow{PU}$.
By symmetry, $\overrightarrow{PA}=\overrightarrow{AU}$.
Since the triangle $AUO_2$ is similar to the triangle $AWO_3$ , we have $2 \overrightarrow{AU}=\overrightarrow{AW}$.
Thus we obtain $\overrightarrow{PU}=2\overrightarrow{AU}=\overrightarrow{AW}$.
On the other hand, the minimum vector of $D_2-D_1$ in $\ell_2$ with respect to the positive orientation of $\ell_2$ is $\overrightarrow{ST}$, where $\overrightarrow{ST}$ can be in the negative direction with respect to $\ell_2$ when $D_1 \cap D_2\neq \emptyset$.
By symmetry, $\overrightarrow{AT}=\overrightarrow{SA}$, and the triangle $ATO_2$ is similar to the triangle $AVO_3$.
Thus $\overrightarrow{ST}=2 \overrightarrow{AT}=\overrightarrow{AV}$.
As a consequence, we obtain $\ell_2 \cap (D_2-D_1)=\ell_2 \cap \{v\,;\, |v-\overrightarrow{O_1O_2}|\leq 2 R\}$.
This is true for any straight line $\ell_2$ passing through $A$ and intersecting $D_2$.
Thus we conclude that $D_2-D_1=\{v\,;\, |v-\overrightarrow{O_1O_2}|\leq 2 R\}$.
\end{proof}

\begin{lemma}
\label{lem4}
$m({Q_c}^{-n}(D)-{Q_c}^{-n}(D)) < 12\pi 4^n {K_n}^2$
\end{lemma}

\begin{proof}
We have 
\begin{eqnarray*}
& &m({Q_c}^{-n}(D)-{Q_c}^{-n}(D))\\
&=&m(\bigcup_{\{ s_0,s_1,\cdots, s_n, {s_0}',{s_1}',\cdots, {s_n}'\}} (I{s_0s_1\cdots s_n}-I{{s_0}'{s_1}'\cdots {s_n}' }))\\
&\leq& \sum_{\{ s_0,s_1,\cdots, s_n, {s_0}',{s_1}',\cdots, {s_n}'\}} m(I{s_0s_1\cdots s_n}-I{{s_0}'{s_1}'\cdots {s_n}' })\\
&\leq& \sum_{\{ s_0,s_1,\cdots, s_n, {s_0}',{s_1}',\cdots, {s_n}'\}} m(D{s_0s_1\cdots s_n}-D{{s_0}'{s_1}'\cdots {s_n}' })
\end{eqnarray*}
(see Figure \ref{figsim}).
\begin{figure}
\label{figsl}
\begin{center}
\includegraphics[width=0.8\linewidth,clip]{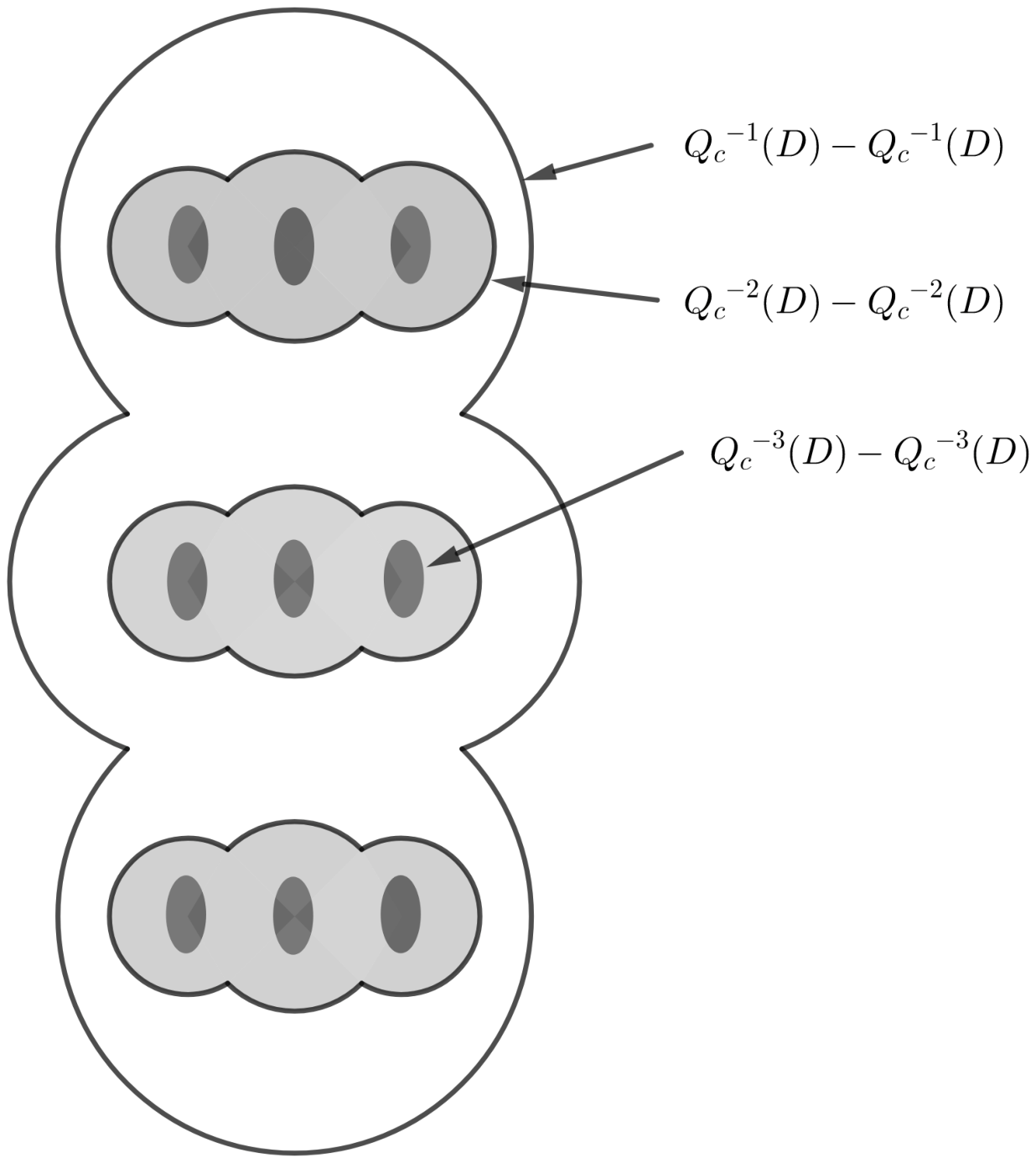}
\end{center}
\caption{${Q_c}^{-n}(D)-{Q_c}^{-n}(D)$}
\end{figure}

Here $\diam I{s_0s_1\cdots s_n} < K_n$ by Lemma \ref{lem1}.
The radius of $D{s_0s_1\cdots s_n}$ is equal to $\frac{\sqrt{3}}{2} \diam I{s_0s_1\cdots s_n}$ by Lemma \ref{lem2}, which is smaller than $\frac{\sqrt{3}}{2}K_n$. 
Furthermore, the radius of $D{{s_0}'{s_1}'\cdots {s_n}' }$ is equal to $\frac{\sqrt{3}}{2}\diam I{{s_0}'{s_1}'\cdots {s_n}' }$, which is also smaller than $\frac{\sqrt{3}}{2}K_n$.
By Lemma \ref{lem3}, 
\begin{eqnarray*}
& &m({Q_c}^{-n}(D)-{Q_c}^{-n}(D))\\
&<& \sum_{\{ s_0,s_1,\cdots, s_n, {s_0}',{s_1}',\cdots, {s_n}'\}} \pi(\sqrt{3}K_n)^2\\
&=& 2^{2n+2} 3 \pi {K_n}^2\\
&=& 12 \pi 4^n {K_n}^2
\end{eqnarray*}
\end{proof}

\noindent {\bf Proof of Theorem \ref{thm1}}

When $|c|>3+\sqrt{3}$, we have
\begin{eqnarray*}
|c|^2-6|c|+6&>&0\\
1+4|c| &<& (2 |c|-5)^2\\
\sqrt{\frac{2|c|-1-\sqrt{1+4|c|}}{2}}&>&\sqrt{2}.
\end{eqnarray*}
Therefore, there is a small $\epsilon>0$ such that $\sqrt{\frac{2|c|-1-\epsilon - \sqrt{1+4|c|}}{2}}>\sqrt{2}$.
Let $\delta$ $(>0)$ denote $\sqrt{\frac{2|c|-1-\epsilon - \sqrt{1+4|c|}}{2}}-\sqrt{2}$.

On the other hand, $\lim_{n\rightarrow \infty} R_n=\frac{1+\sqrt{1+4|c|}}{2}$ and $\lim_{n\to \infty}r_n=\sqrt{\frac{2|c|-1- \sqrt{1+4|c|}}{2}}$.
Since $r_n$ is monotone increasing, there exists a positive integer $N$ such that, if $n>N$, then $r_n \geq 
\sqrt{\frac{2|c|-1-\epsilon-\sqrt{1+4|c|}}{2}}(=\sqrt{2}+\delta)$.

By Lemma \ref{lem1}, $\diam I{s_0s_1\cdots s_n} < K_n$, where $K_n=2^{-\frac{n}{2}} \frac{1}{r_{n+1}r_n\cdots r_2} \diam I_0$.
For $n>N$, we obtain 
\begin{eqnarray*}
& & m({Q_c}^{-n}(D)-{Q_c}^{-n}(D))\\
&<& 12 \pi 4^n {K_n}^2\\
&=& 12 \pi 4^n \frac{1}{2^n} \left( \frac{1}{r_{n+1}r_n\cdots r_{N+1}} \right)^2 \left(\frac{1}{r_Nr_{N-1}\cdots r_2}\right)^2 \diam^2 I_0\\
&\leq & 12 \pi 4^n \frac{1}{2^n} \frac{1}{(\sqrt{2}+\delta)^{2(n+1-N)}} \left( \frac{1}{r_N r_{N-1}\cdots r_2}\right)^2 \diam ^2 I_0.
\end{eqnarray*}
Thus there is a constant $K>0$ such that $m({Q_c}^{-n}(D)-{Q_c}^{-n}(D))\leq K \left(  \frac{\sqrt{2}}{\sqrt{2}+\delta} \right)^{2n}$ for $n>N$.
As a consequence, $m(J_c - J_c)=\lim_{n \to \infty} m({Q_c}^{-n}(D)-{Q_c}^{-n}(D)) =0$.
\hfill \qedsymbol



\noindent
Hiromichi Nakayama\\
E-mail: nakayama@gem.aoyama.ac.jp\\
Takuya Takahashi\\
Department of Physics and Mathematics, \\
College of Science and Engineering, \\
Aoyama Gakuin University, \\
5-10-1 Fuchinobe, Sagamihara, Kanagawa, 252-5258, Japan\\

\end{document}